\documentclass{amsart}

\usepackage{amssymb}
\usepackage{overpic}
\usepackage{enumitem}   
\usepackage{graphicx} 
\usepackage{amsrefs}
\usepackage{xcolor}
\usepackage[hidelinks]{hyperref}

\graphicspath{{Figures/}} %To find the file with the figures

\newtheorem{corollary}{Corollary} 
 
\newtheorem{lemma}{Lemma} 
\newtheorem{definition}{Definition} 
\newtheorem{example}{Example}

\newtheorem{main}{Theorem} %Main theorems, labelled with letters
\begin{document}
	
\title[Homoclinic Theorems for piecewise smooth vector fields]{Homoclinic Theorems for piecewise smooth vector fields}

\author[Claudio Buzzi, Paulo Santana, and Luan V. M. F. Silva]{Claudio Buzzi$^1$, Paulo Santana$^1$, and Luan V. M. F. Silva$^2$}

\address{$^1$IBILCE--UNESP, CEP 15054--000, S. J. Rio Preto, S\~ao Paulo, Brazil}
\email{claudio.buzzi@unesp.br; paulo.santana@unesp.br}

\address{$^2$IFTM, CEP 38305-200, Ituiutaba, Minas Gerais, Brazil}
\email{luan@iftm.edu.br}



\begin{abstract}
	In this paper we provide extensions of the Birkhoff-Smale Homoclinic Theorem and the $\lambda$-Lemma for piecewise smooth vector fields free of sliding motion. We prove that if the vector field is $T$-periodic in its independent variable and has a transverse homoclinic point, then its time-$T$-map is conjugated with a Bernoulli Shift and thus the vector field is chaotic. The Bernoulli Shifts considered here may have any finite number of symbols, or countably infinite many.
\end{abstract}

\maketitle

\section{Introduction}%\label{Sec1}

When studying the 3-body problem, Poincaré~\cite[Chapter~$33$]{Poincare} observed the occurrence of transversal intersections between the stable and unstable manifolds of a hyperbolic point. Such an intersection was named a \emph{homocline} point~\cite[p.~$384$]{Poincare} and it was proved that the dynamics around it are quite rich, as for example one homocline point implies the existence of infinitely many~\cite[p.~$387$]{Poincare}. Nowadays such a point is known as \emph{homoclinic} point and plays a key role for chaos. 

The relation between homoclinic points and chaos was explored by Birkhoff~\cite[Chapter~$4$]{Birkhoff}, where it was proved for two-dimensional $C^1$-diffeomorphisms, that a homoclinic point of non-flat contact between the stable and unstable manifolds is accumulated by periodic points of increasing period~\cite[p.~$178$]{Birkhoff}. Birkhoff also claimed the existence of richer dynamics, with an \emph{``almost inconceivable hierarchy of solutions''}~\cite[p.~$184$ and~$185$]{Birkhoff}, including not only the periodic points, but also the asymptotic solutions to those periodic points (forward or backward in time), recurrent trajectories and their respective asymptotics.

This culminates in the seminal contribution of Smale~\cite[Theorem~$B$]{SmalePaper}, who proved that in a neighborhood of a transversal homoclinic point of a $n$-dimensional $C^1$-diffeomorphism there is an invariant Cantor set whose dynamics are topologically conjugated to a Bernoulli Shift of \emph{finitely} many symbols. Smale also proved that such a dynamics persists under small perturbations~\cite[Theorem~$A$]{SmalePaper}. These results are nowadays know as \emph{Birkhoff-Smale Theorem}~\cite[Theorem~$5.3.5$]{GukHol1983}, or \emph{Smale Homoclinic Theorem}~\cite[Theorem~$2.3$]{Newhouse}, and they were the foundation of the groundbreaking \emph{horseshoe map}~\cite[p.~$771$]{SmaleSurvey}. 

After that, Palis~\cite[Lemma~$1.1$]{Palis} proved that the stable and unstable manifolds of the transverse homoclinic point accumulate over themselves in a $C^1$-manner. This result is now known as \emph{$\lambda$-Lemma} and is the foundation of the notion of \emph{homoclinic tangle}. Besides that, the $\lambda$-Lemma became a keystone in the proof of an extension of Birkhoff-Smale Theorem regarding the hyperbolicity of the invariant set, meaning that every point in the invariant set has local stable and unstable manifolds that are carried with the base-point when iterating the diffeomorphism. Nowadays it is know that invariant set provided by the Birkhoff-Smale Theorem is hyperbolic, and it is common to include this information in modern versions of the theorem, see~\cite[Theorem~$5.3.5$]{GukHol1983} or~\cite[Theorem~$2.3$]{Newhouse}.

There is another extension for the Birkhoff-Smale Theorem, provided as far as we know by Moser~\cite[Theorem~$3.7$]{Moser}, although he says having learned it from Conley~\cite[p.~$100$]{Moser} (without mentioning a reference). Moser's Theorem is two-folded. First it is proved that there is an invariant set (not compact) whose dynamics are topologically conjugated to a Bernoulli Shift of \emph{countably infinite} many symbols. Then it is proved that this conjugation can be extended to the compactification of the invariant set, in order to include the points that ``escape'' for positive or negative time~\cite[p.~$75$]{Moser}. In particular, it is proved that the homoclinic points are dense~\cite[Theorem~$3.8$]{Moser} in this compactification.

From there on the theory flourish and generalizations appeared, such as for non-transversal homoclinic points~\cite{Rayskin} (and the references therein), normally hyperbolic manifolds~\cite{CreWig2015}, Lipschitz functions~\cite{Lip}, and piecewise smooth vector fields~\cites{BatFec2010,BatFec2011,CalFraPop2025}. Regarding the latter it is known that in the presence of sliding motion, a transversal homoclinic point does not guarantee chaos~\cite{FraPos2019}. 

However, as far as we know, only the original version of Birkhoff-Smale Theorem was provided for piecewise smooth vector field without the presence of sliding~\cite{BatFec2011}. In this paper we provide a new proof for this result, but that also includes the hyperbolicity of the invariant set, and the generalizations of Moser. A proof of the $\lambda$-Lemma is also provided, allowing, for example, the reduction of heteroclinic intersections to homoclinic ones. 

The paper is organized as follows. We begin with a preliminary Section~\ref{Sec2} for piecewise smooth vector fields, where some technical lemmas are stated. The gene\-ra\-lizations of the $\lambda$-Lemma and the Homoclinic Theorems are provided in Section~\ref{Sec3}. In Section~\ref{Sec4} we obtain a ``lift'' of these results for the piecewise smooth vector field itself, and apply it on an example. The technical lemmas are proved in Section~\ref{Sec5}. 

\section{Preliminary: Piecewise smooth vector fields}\label{Sec2}

Let $h_1,\dots,h_N\colon\mathbb{R}^n\to\mathbb{R}$ be $C^1$-functions, $\Sigma_i=h^{-1}_i(\{0\})$, and $\Sigma=\cup_{i=1}^{N}\Sigma_i$. Let $A_1,\dots,A_M$ be the open connected components of $\mathbb{R}^n\setminus\Sigma$, and $X_j$ an autonomous $C^1$-vector field defined in $\overline{A_j}$, $j\in\{1,\dots,M\}$, with the bar denoting the topological closure of a set. The piecewise smooth vector field is defined by
\begin{equation}\label{01}
	Z(x)=\left\{\begin{array}{ll}
		X_1(x), & \text{if } x\in \overline{A_1}, \\
		\dots & \\
		X_M(x), & \text{if } x\in \overline{A_M}.
	\end{array}\right.
\end{equation}
We denote system~\eqref{01} by $Z=(X_1,\dots,X_M,\Sigma)$. We say that each $X_j$ is a \emph{component} of $Z$, while $\Sigma$ is its \emph{switching set}. Observe that $Z$ is multi-valued on $\Sigma$. Thus, a transition rule among the components of $Z$ across $\Sigma$ is necessary. Here we adopt Filippov's convention~\cite{FilippovBook} and its generalization~\cite{PanSil2017} for such rules. 

An important definition in Filippov's convention is that of \emph{crossing points}. To detect such points, let us introduce the derivative of $h_i$ with respect to $X_j$ as follows:
\[
	X_jh_i(x):=\left<X_j(x),\nabla h_i(x)\right>,
\]
where $\left<\cdot,\cdot\right>$ denotes the standard inner product of $\mathbb{R}^n$. We say that $x\in\Sigma$ is a \emph{crossing point} of $Z$ if the following conditions are satisfied.
\begin{enumerate}[label=(\roman*)]
	\item\label{i} There is a unique $i\in\{1,\dots,N\}$ such that $x\in\Sigma_i$.
	\item\label{ii} $\nabla h_i(x)\neq0$.
	\item\label{iii} $X_ah_i(x)X_bh_i(x)>0$, with $X_a$, $X_b$ the two components of $Z$ defined in $x$.
\end{enumerate}
In order to define a trajectory (i.e. a solution) of~\eqref{01} through a crossing point $x\in\Sigma$, it is enough to concatenate the trajectories of $X_a$ and $X_b$ through the point, where $X_a$ and $X_b$ are the only two components of $Z$ defined over $x$. 

Let $\Sigma_c\subset\Sigma$ be the subset of crossing points of $Z$. If $x\in\Sigma\setminus\Sigma_c$, then $Z$ may have multiple different trajectories defined through it. Since we will not work with such points in this paper, we will not dive into it. Given $x\in(\mathbb{R}^n\setminus\Sigma)\cup\Sigma_c$, let $\varphi(t,x)$ be the local trajectory of $Z$ through $x$, satisfying $\varphi(0,x)=x$. Observe that $\varphi$ can be uniquely extended as long it does not intersect $\Sigma\setminus\Sigma_c$. Hence, from now $\varphi(t,x)$ is defined as the maximum trajectory of $Z$ through $x$, that does not intersect $\Sigma\setminus\Sigma_c$. 

\begin{lemma}\label{P1}
	Let $x\in\mathbb{R}^n\setminus\Sigma$ and $T>0$ such that $\{\varphi(t,x)\colon t\in[0,T]\}$ intersects $\Sigma$ only in crossing points. Then such an intersection is finite.
\end{lemma}

Notice if $x\in\Sigma_c$, then there is a neighborhood $A$ of $x$ such that $(A\cap\Sigma)\subset\Sigma_c$. Hence, it follows from Lemma~\ref{P1} that if $x\in\mathbb{R}^n\setminus\Sigma$ is such that $\{\varphi(t,x)\colon t\in[0,T]\}$ intersects $\Sigma$ only in crossing points, then there is a neighborhood $A\subset\mathbb{R}^n\setminus\Sigma$ of $x$ such that for every $r\in A$, $\varphi(t,r)$ is defined for $t\in[0,T]$, and intersects $\Sigma$ only in crossing points. This allows the definition of the time-$T$-map $\Phi^T\colon A\to\mathbb{R}^n$, given by $\Phi^T(r)=\varphi(T,r)$.

\begin{lemma}\label{P2}
	Let $x\in\mathbb{R}^n\setminus\Sigma$ and $T>0$ such that $\{\varphi(t,x):t\in[0,T]\}$ intersects $\Sigma$ only in crossing points, and $\varphi(T,x)\not\in\Sigma$. Then $\Phi^T$ is a $C^1$-diffeomorphism in a neighborhood of $x$.
\end{lemma}

We now show that the hypotheses $x\not\in\Sigma$ and $\varphi(T,x)\not\in\Sigma$ are essential.

\begin{example}
	Let $Z=(X^+,X^-,\Sigma)$, with $X^+(x,y)=(0,1)$, $X^-(x,y)=(1,1)$, and $\Sigma=\{(x,y)\in\mathbb{R}^2\colon y=0\}$. The solutions of $X^\pm$ are given by
	\begin{equation*}%\label{2}
		\varphi^+(t,x,y)=(x,t+y)\quad \mbox{and} \quad  \varphi^-(t,x,y)=(t+x,t+y).
	\end{equation*}
	It is not hard to see that the solution $\varphi(t,x,y)$ of $Z$ is given by
	\begin{equation}\label{3}
		\varphi(t,x,y)=\left\{\begin{array}{ll}
			(t+x,t+y) & \text{if } y\leqslant0 \text{ and } t\leqslant -y, \\
			(x-y,t+y) & \text{if } y\leqslant0 \text{ and } t\geqslant -y, \\
			(x,t+y) & \text{if } y>0 \text{ and } t\geqslant -y, \\
			(t+x+y,t+y) & \text{if } y>0 \text{ and } t\leqslant -y.
		\end{array}\right.
	\end{equation}
	Given $T>0$, it follows from \eqref{3} that
	\[
		\lim\limits_{y\to0^+}\frac{\Phi^T(0,y)-\Phi^T(0,0)}{y}=(0,1), \quad \lim\limits_{y\to0^-}\frac{\Phi^T(0,y)-\Phi^T(0,0)}{y}=(-1,1).
	\]
	Hence $\Phi^T$ is not differentiable at the origin, for  every $T>0$.
\end{example}

Notice that even if $\Phi^T$ is not of class $C^1$ everywhere, it will be continuous.

\begin{lemma}\label{P3}
	Let $x\in(\mathbb{R}^n\setminus\Sigma)\cup\Sigma_c$, and $T>0$ such that $\{\varphi(t,x):t\in[0,T]\}$ intersects $\Sigma$ only in crossing points. Then $\Phi^T$ is a homeomorphism in a neighborhood of $x$.
\end{lemma}

The proofs of Lemmas~\ref{P1},~\ref{P2} and~\ref{P3} are postponed to Section~\ref{Sec5}.

\section{Piecewise $\lambda$-Lemma, Homoclinic Theorems and Chaos}\label{Sec3}

Let $\mathcal{U}$, $\mathcal{V}\subset\mathbb{R}^n$ be open sets, $\Omega_1,\dots,\Omega_N\subset\mathbb{R}^n$ immersed sub-manifolds of codimension one, and $\Omega=\cup_{i=1}^{N}\Omega_i$. 
\begin{enumerate}[label=$(H)$]
	\item\label{H} We say that a homeomorphism $F\colon\mathcal{U}\to \mathcal{V}$ satisfies hypothesis $(H)$ if the following conditions are satisfied.
	\begin{enumerate}[label=(\roman*)]
		\item $F$ is of class $C^1$ in $ \mathcal{U}\setminus\Omega$.
		\item If $p\in\mathcal{U}\setminus\Omega$, $\{F(p),\dots,F^{k-1}(p)\}\subset \mathcal{U}$, and $F^k(p)\in\mathcal{V}\setminus\Omega$, then $F^k$ is diffeomorphism of class $C^1$ in a neighborhood of $p$.
		\item $F^{-1}$ satisfies properties $(i)$ and $(ii)$, interchanging $ \mathcal{U}$ and $ \mathcal{V}$.
	\end{enumerate}
\end{enumerate}
We remark from Lemmas~\ref{P2} and~\ref{P3} that the time-$T$-map $\Phi^T$ of a piecewise smooth vector field $Z$ free of sliding motion satisfies hypothesis~\ref{H}, which is the goal of this paper. We introduced such a hypothesis to simplify the statements the results. 

\subsection{Piecewise $\lambda$-Lemma}

We say that a fixed point $p\in\mathcal{U}\setminus\Omega$ of $F$ is \emph{hyperbolic} if none of the eigenvalues of $DF(p)$ has norm equal to $1$. In this case, it follows from the Stable Manifold Theorem~\cite[Theorems~$2.3$ and~$2.4$]{HirPugh} that $p$ has well-defined local stable and unstable $C^1$-manifolds $W^s_{loc}(p)$, $W^u_{loc}(p)$ with transversal intersection in $p$, and whose dimensions $s\geqslant0$ and $u\geqslant0$ satisfies $s+u=n$. The global stable and unstable manifolds are given by $W^s(p)=\bigcup_{i\leqslant0}F^i\big(W^s_{loc}(p)\big)$, and $W^u(p)=\bigcup_{i\geqslant0}F^i\big(W^u_{loc}(p)\big)$. Notice from Lemma~\ref{P1} that $W^{s,u}(p)$ are of class $C^1$ almost everywhere.

Given $D\subset\mathbb{R}^n$, we say that it is a \emph{$k$-disk} if it is diffeomorphic to a closed ball of $\mathbb{R}^k$. If $D$ is diffeomophic to an open ball of $\mathbb{R}^k$ and $r\in D$, then we say that it is a $k$-neighborhood of $r$.

\begin{main}[Piecewise $\lambda$-Lemma]\label{M1}
	Let $F$ satisfy hypothesis~\ref{H} and $p\in\mathcal{U}\setminus\Omega$ be a hyperbolic fixed point of $F$. Let $\Delta$ be a $k$-disk intersecting $W^s(p)$ transversely at some point $q\in\mathcal{U}\setminus\Omega$. Then, for every $u$-disk $D\subset W^u(p)$ the set $\cup_{i\geqslant0}F^i(\Delta)$ contains $u$-disks arbitrarily close in the $C^1$-topology (resp. $C^0$-topology) to $D$ when $D\cap\Omega=\emptyset$ (resp. $D\cap\Omega\neq\emptyset$).
\end{main}

It follows from Theorem~\ref{M1} that if $\Delta\subset\mathcal{U}$ is a sub-manifold intersecting $W^s(p)$ transversely outside $\Omega$, then $W^u(p)\subset\overline{\cup_{i\geqslant0}F^i(\Delta)}$. Hence, if $W\subset\mathcal{U}$ is an invariant sub-manifold such that $\Delta\subset W$, then $W$ accumulates at $W^u(p)$, see Figure~\ref{Fig6}$(a)$. 
\begin{figure}[ht]
	\begin{center}
		\begin{minipage}{6cm}
			\begin{center}
				\begin{overpic}[height=4cm]{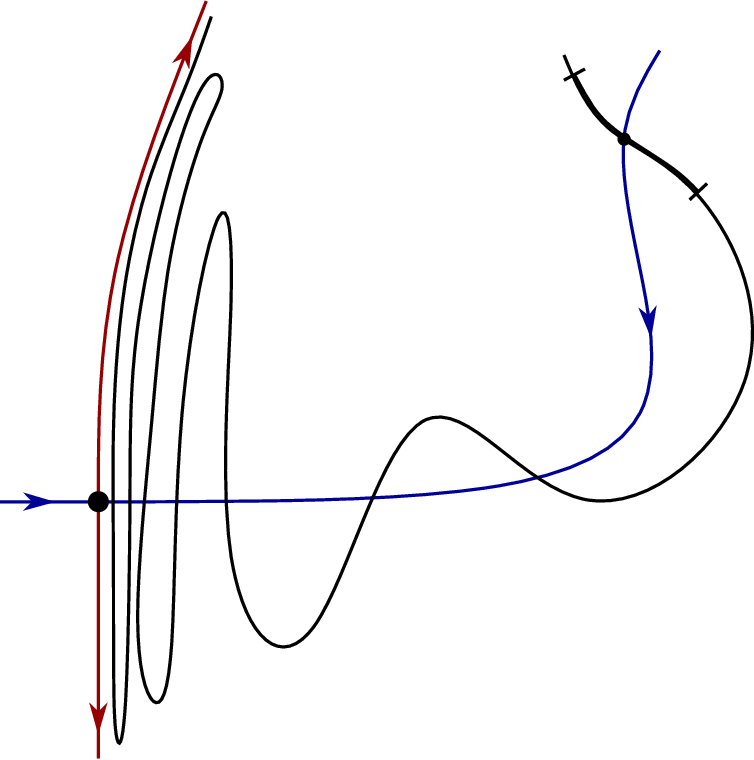} 
				%\begin{overpic}[height=4cm,grid,tics=5]{Fig8.eps}
					\put(7,28){$p$}
					\put(76,77){$q$}
					\put(91,77){$\Delta$}
					\put(59,55){$W^s(p)$}
					\put(-10,70){$W^u(p)$}
					\put(44,15){$W$}
				\end{overpic}	
							
				$(a)$
			\end{center}
		\end{minipage}
		\begin{minipage}{6cm}
			\begin{center}
				\begin{overpic}[height=4cm]{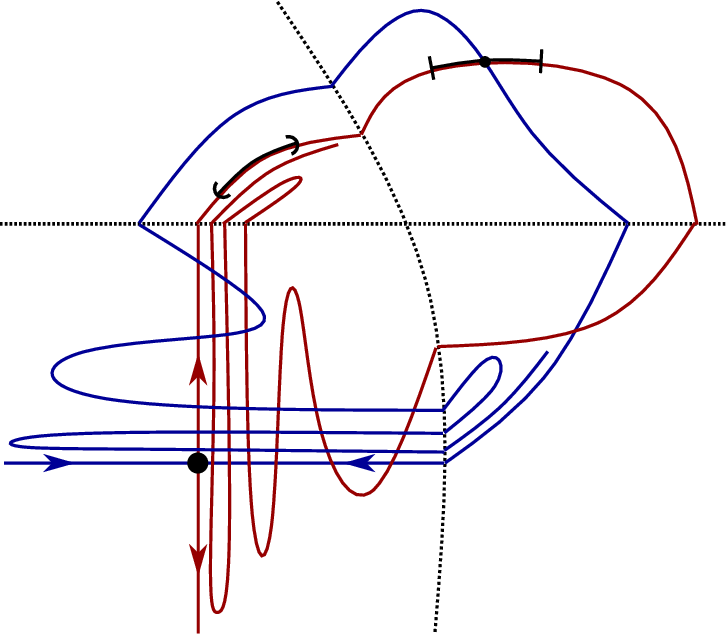} 
				%\begin{overpic}[height=cm,grid,tics=5]{Fig15.eps}
					\put(22,18){$p$}
					\put(63,74){$q$}
					\put(3,65){$W^s(p)$}
					\put(89,42){$W^u(p)$}
					\put(57,57.75){$\Omega$}
					\put(72,81){$\Delta$}
					\put(31,66){$D$}
				\end{overpic}
								
				$(b)$
			\end{center}
		\end{minipage}
	\end{center}
\caption{Illustration of Theorem~\ref{M1}.}\label{Fig6}
\end{figure}
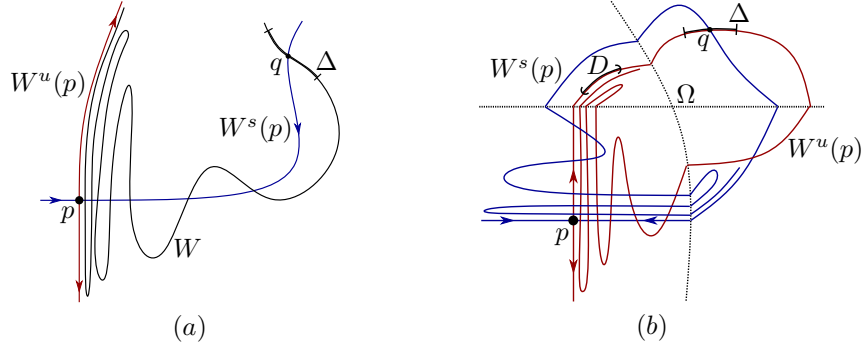 
In particular, if $W=W^u(p)$, then it occurs a \emph{homoclinic tangle}, see Figure~\ref{Fig6}$(b)$.

\begin{proof}[Proof of Theorem~\ref{M1}]
The bulk of the proof of Theorem~\ref{M1} is local. First we consider a small enough neighborhood $U\subset\mathcal{U}\setminus\Omega$ in which coordinates of the Stable Manifold Theorem can be assumed. Then, we iterate $q$ sufficiently many times, say $k$ times, until it enters $U$, see Figure~\ref{Fig3}.
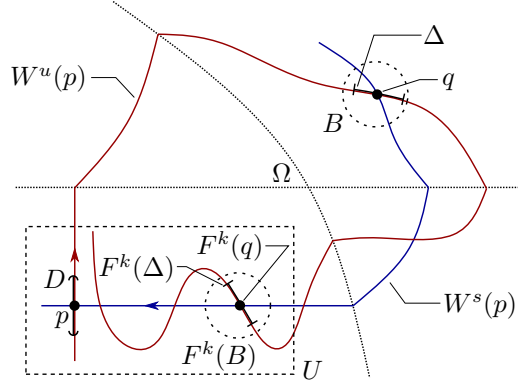
\begin{figure}[ht]
	\begin{center}
		\begin{overpic}[height=5cm]{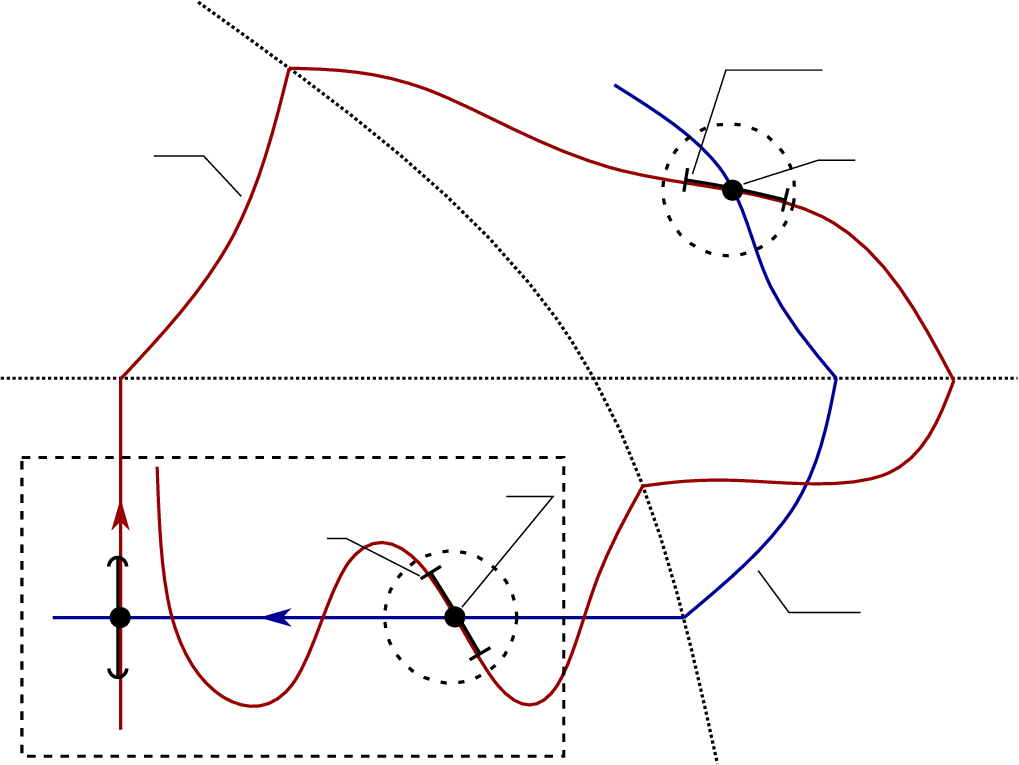} 
			%\begin{overpic}[height=5cm,grid,tics=5]{Fig4.eps} 
			\put(84.75,58.25){$q$}
			\put(61,49){$B$}
			\put(81,66.5){$\Delta$}
			\put(51,39){$\Omega$}
			\put(33,3){$F^k(B)$}
			\put(36,24.75){$F^k(q)$}
			\put(16.5,20.25){$F^k(\Delta)$}
			\put(8,11){$p$}
			\put(5.5,18){$D$}
			\put(57,0){$U$}
			\put(-1,58.25){$W^u(p)$}
			\put(85,13.25){$W^s(p)$}
		\end{overpic}
	\end{center}
	\caption{Illustration of the proof of Theorem~\ref{M1}.}\label{Fig3}
\end{figure}
Since $q$, $F^k(q)\not\in\Omega$ we have from hypothesis~\ref{H} that $F^k$ is a $C^1$-diffeomorphism in a neighborhood $B$ of $q$. Hence, we can assume without loss of generality that the transverse intersection occurs inside $U$. One now apply the proof of the $\lambda$-Lemma~\cite[Lemma~$7.1$]{PalisBook}, obtaining that the iterates $F^i(\Delta)$ gets arbitrarily $C^1$-close to a $u$-disk $D\subset W^u(p)\subset U$. The general case where $D\subset W^u(p)$ is any $u$-disk, is as follows. Let $D\subset W^u(p)$ be a $u$-disk and observe that the iterates $F^{-m}(D)$ get arbitrarily close to $p$. If $D\cap\Omega=\emptyset$, then $F^{-m}$ is a $C^1$-diffeomorphism in a neighborhood of $D$ for $m\in\mathbb{N}$ big enough. From the previous case we can take $F^i(\Delta)$ arbitrarily $C^1$-close to $F^{-m}(D)$ and thus $F^{n+m}(\Delta)$ contain a $u$-disk arbitrarily $C^1$-close to $D$. If $D\cap\Omega\neq\emptyset$, then $F^m$ may not be of class $C^1$, but is still a homeomorphism and thus we can approximate in the $C^0$-topology. 
\end{proof}

The $\lambda$-Lemma is also used to reduce a heroclinic intersection to a homoclinic one. Given two hyperbolic fixed points $p_1$, $p_2\in\mathcal{U}\setminus\Omega$, we say that $p_1$ is related with $p_2$ if $W^u(p_1)$ has a transversal intersection with $W^s(p_2)$ outside $\Omega$. We denote such a relation by $p_1\sim p_2$. Notice that if $p_1\sim p_2$ and $p_2\sim p_3$, then it follows directly from the $C^1$-approximation provided by Theorem~\ref{M1} that $p_1\sim p_3$. See Figure~\ref{Fig11} and also~\cite[p.~$85$, Corollary~$1$]{PalisBook}.
\begin{figure}[ht]
	\begin{center}
		\begin{overpic}[height=4cm]{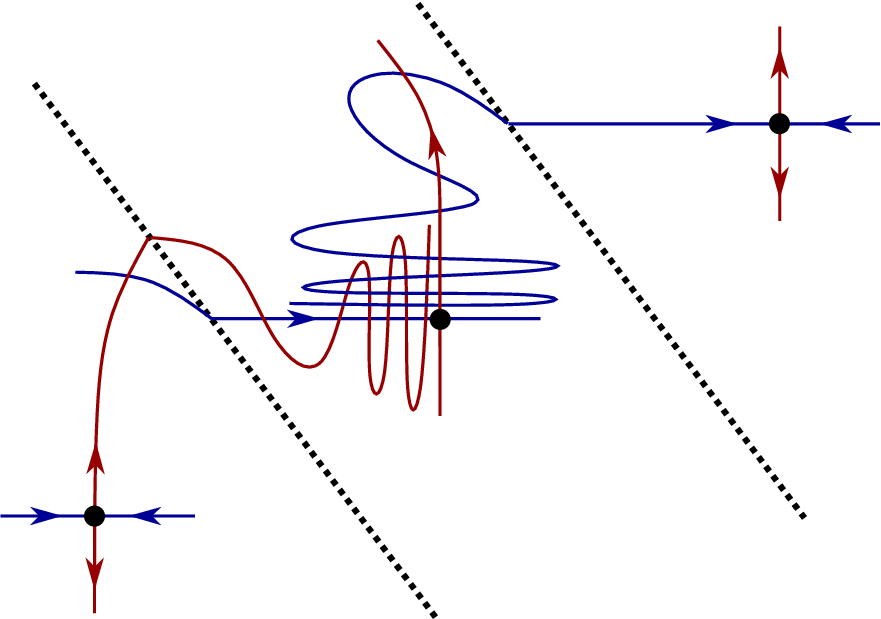} 
			%\begin{overpic}[height=4cm,grid,tics=5]{Fig14x.eps}
			\put(13,7){$p_1$}
			\put(52,30){$p_2$}
			\put(90.5,52){$p_3$}
			\put(51,-1){$\Omega$}
			\put(93,10){$\Omega$}
		\end{overpic}
	\end{center}
	\caption{Illustration of the transitivity of transversal homoclinic intersections.}\label{Fig11}
\end{figure}
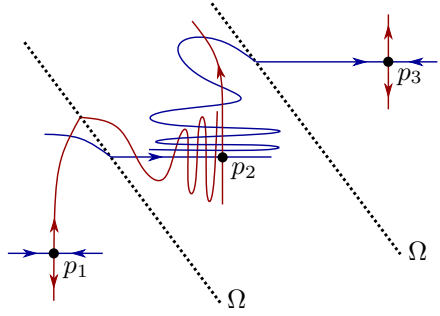 
Therefore, we have the following corollary.

\begin{corollary}%\label{C1}
	Let $F$ satisfy hypothesis~\ref{H} and $p_1$, $p_2$, $p_3\in\mathcal{U}\setminus\Omega$ be hyperbolic fixed points of $F$. If $p_1\sim p_2$ and $p_2\sim p_3$, then $p_1\sim p_3$.
\end{corollary}

It particular, observe that if $p_1\sim p_2$, $p_2\sim p_3$ and $p_3\sim p_1$, then $p_i\sim p_i$ for every $i\in\{1,2,3\}$. That is, the existence of a polycycle composed by $p_1$, $p_2$ and $p_3$, implies that each $p_i$ has a transverse homoclinic intersection.

\subsection{Piecewise Birkhoff-Smale Theorem}

Let $S_j$ (resp. $S_\infty$) be the space of bi-infinite sequences $s=\{s_k\}_{k\in\mathbb{Z}}$ with $s_k\in I_j$, where $I_j=\{1,\dots,j\}$, $j\geqslant2$ (resp. $I_\infty=\mathbb{N}$), endowed with the product topology. The \emph{Bernoulli Shift} is the homeomorphism $\sigma\colon S_j\to S_j$ (resp. $\sigma\colon S_\infty\to S_\infty)$ given by $\sigma(s)_k=s_{k+1}$. That is,
\begin{equation}\label{02}
	\sigma(\dots,s_{-1},s_0;s_1,\dots)=(\dots,s_0,s_1;s_2,\dots).
\end{equation}
We say that a set $\Lambda\subset\mathcal{U}$ is \emph{invariant} by $F$ if $F(\Lambda)=\Lambda$. As previously mentioned, a compact invariant set is \emph{hyperbolic} if every point has local stable and unstable manifolds which are carried with the base-point when iterating the map. Since we do not need to work directly with it in this paper, we omit its rigorous definition. We refer the interested reader to~\cite[Def.~$5.2.6$ and Theo.~$5.2.8$]{GukHol1983}. Given a hyperbolic fixed point $p\in\mathcal{U}\setminus\Omega$, we say that $q\in\mathcal{U}\setminus\Omega$, $q\neq p$, is a \emph{transverse homoclinic point} if $W^s(p)$, $W^u(p)$ intersects transversely at $q$. 

\begin{main}[Piecewise Birkhoff-Smale Theorem]\label{M2}
	Let $F$ satisfy hypothesis~\ref{H}, $p\in\mathcal{U}\setminus\Omega$ be a hyperbolic fixed point of $F$, and $q\in\mathcal{U}\setminus\Omega$ a transverse homoclinic point. Then for each $j\in\mathbb{N}_{\geqslant2}$, there is $k\in\mathbb{N}$ and a homeomorphism $\tau_j\colon S_j\to\tau(S_j)\subset\mathcal{U}\setminus\Omega$, such that $\tau_j(S_j)$ is compact, invariant by $F^k$, hyperbolic, and $F^k\circ\tau_j=\tau_j\circ\sigma$.
\end{main}

\begin{proof}[Proof of Theorem~\ref{M2}]
Similarly to the $\lambda$-Lemma, the bulk of the proof is local. First we iterate the homoclinic point $q$ until entering a neighborhood $U\subset\mathcal{U}\setminus\Omega$ of $p$ in which coordinates of the Stable Manifold Theorem can be assumed. Then, we consider a small neighborhood $N\subset U$ of $W^s(p)\cap U$ containing $p$ and $q$ in its interior, and let $k\in\mathbb{N}$ be big enough such that $F^k(N)$ resembles the horseshoe map, see Figure~\ref{Fig15}$(a)$.
\begin{figure}[ht]
	\begin{center}
		\begin{minipage}{6.25cm}
			\begin{center}
				\begin{overpic}[width=5.75cm]{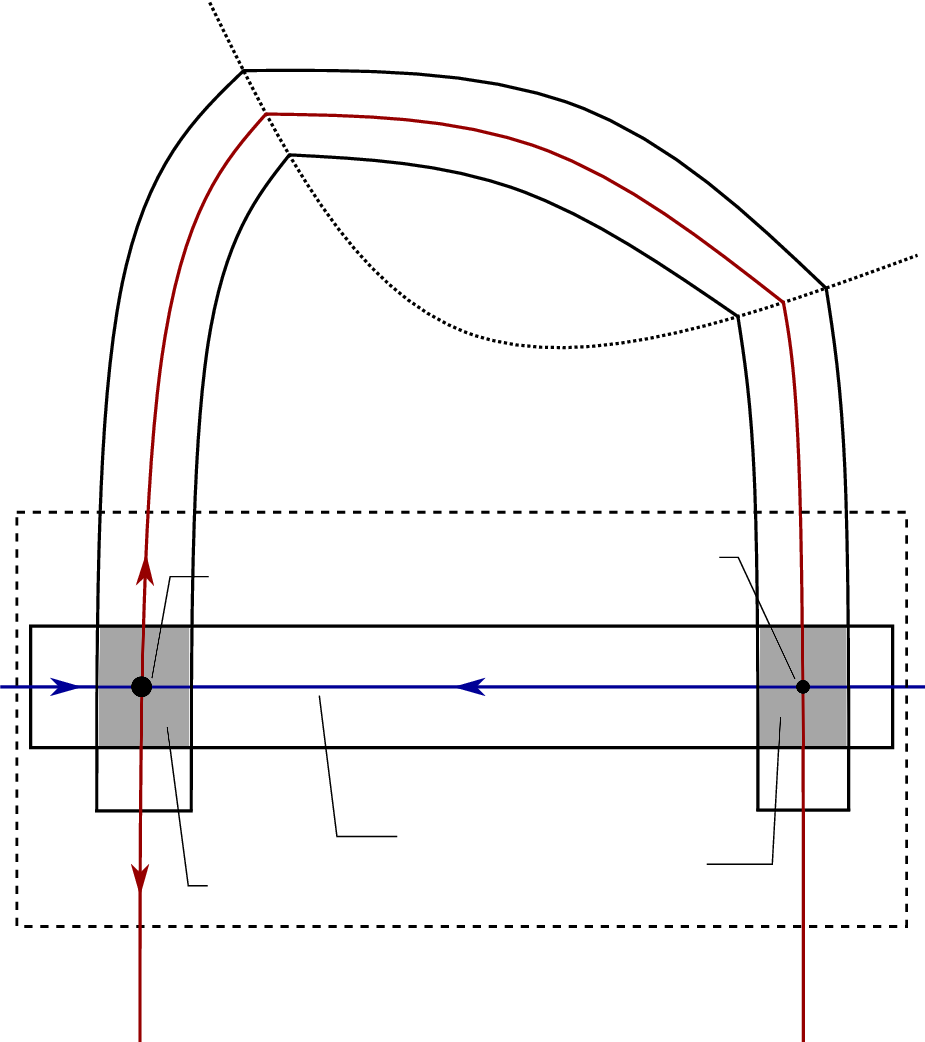} 
				%\begin{overpic}[width=5.75cm,grid,tics=5]{Fig20.eps}
					\put(20.5,43.75){$p$}
					\put(66,45.5){$q$}
					\put(39,18){$W^s(p)$}
					\put(14.5,0){$W^u(p)$}
					\put(85,76){$\Omega$}
					\put(20.5,13.5){$V_1$}
					\put(62,15){$V_2$}
					\put(1,52){$U$}
					\put(43.5,41){$N$}
					\put(61,88){$F^k(N)$}
				\end{overpic}
					
				$(a)$
			\end{center}
		\end{minipage}
		\begin{minipage}{6.25cm}
			\begin{center}
				\begin{overpic}[width=5.75cm]{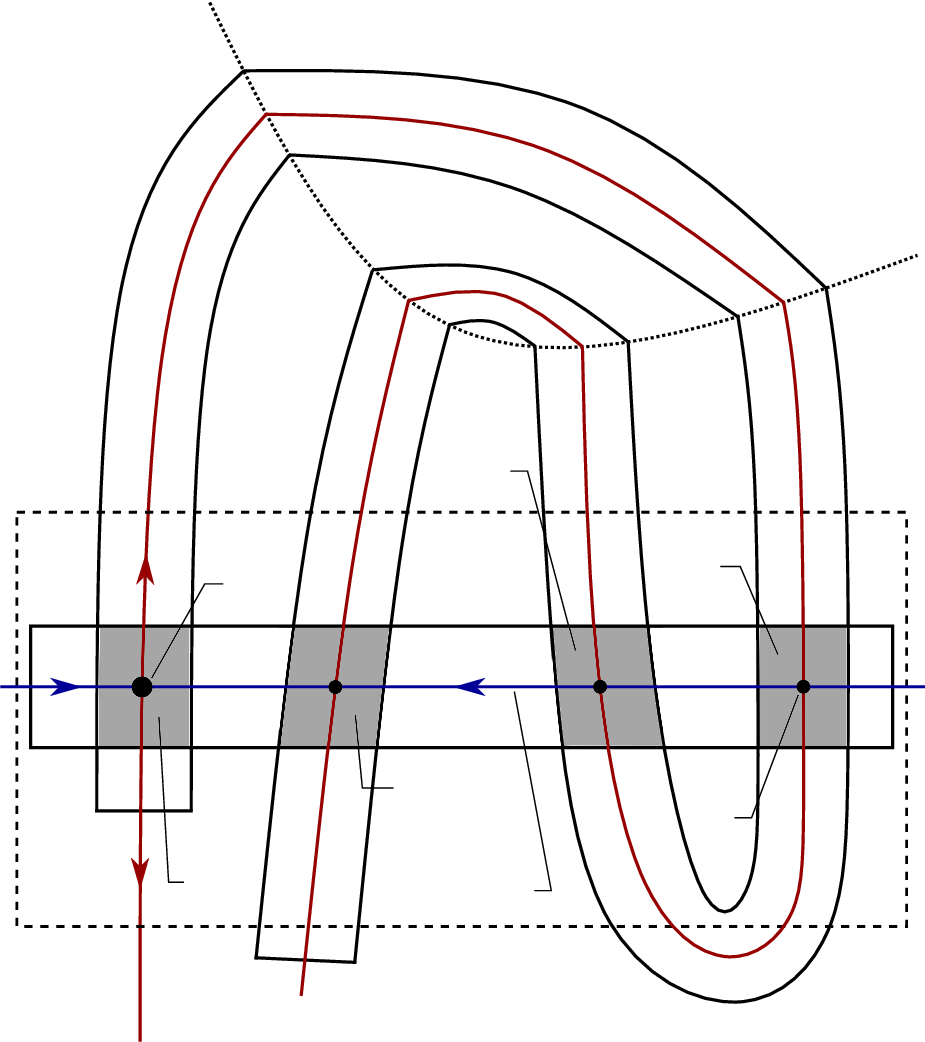} 
				%\begin{overpic}[width=5.75cm,grid,tics=5]{Fig21.eps}
					\put(22,43){$p$}
					\put(67.5,20.5){$q$}
					\put(36,13){$W^s(p)$}
					\put(-2.5,0){$W^u(p)$}
					\put(85,76){$\Omega$}
					\put(18,13.75){$V_1$}
					\put(63.25,43.75){$V_2$}
					\put(1,52){$U$}
					\put(43,52.75){$V_3$}
					\put(38,22.5){$V_4$}
					\put(43.5,41){$N$}
					\put(61,88){$F^k(N)$}
				\end{overpic}
					
				$(b)$ 
			\end{center}
		\end{minipage}
	\end{center}
\caption{Illustration of Theorem~\ref{M2} with $(a)$ $j=2$ and $(b)$ $j=4$.}\label{Fig15}
\end{figure}
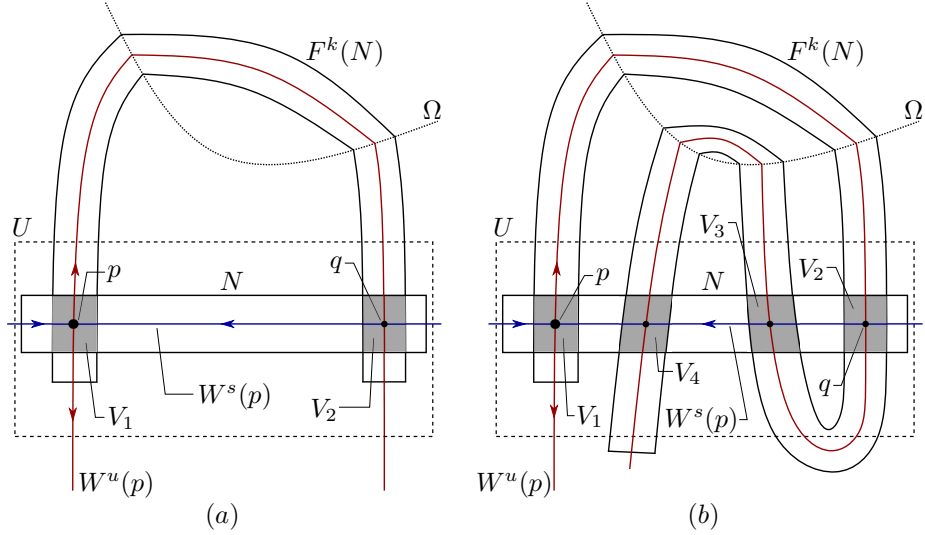 
We then let $V_i$ be the connected components of $N\cap F^k(N)$ containing $p$, $q$, and other iterates of $q$, if $j>2$. The map $\tau_j\colon S_j\to\Lambda_j$ takes $s=(s_m)_{m\in\mathbb{Z}}$ and sends it to a point $r$ such that $F^{km}(r)\in V_{s_m}$ for every $m\in\mathbb{Z}$. Let $H_i=F^{-k}(V_i)$. Since $H_i$, $V_i\subset U$ and $U\subset\mathcal{U}\setminus\Omega$, it follows from hypothesis~\ref{H} that $F^k|_{H_i}$ is a $C^1$-diffeomorphism, $i\in\{1,\dots,j\}$. Hence, one can apply the usual proof of the theorem, see for instance~\cite[Theorem~$2.3$]{Newhouse}.	
\end{proof}

\subsection{Piecewise version of Moser's Theorem}\label{Sec3.3}

We now work on the case of countably infinite symbols. It follows from Tychonoff Theorem~\cite[Theorem~$37.3$]{Munkres} that the space of sequence of $j$-symbols $S_j$ is compact, for each $j\in\mathbb{N}$. However, this is not the case for $S_\infty$. Hence, we introduce a compactification $\overline{S_\infty}$ of $S_\infty$ by including elements of the following type, see~\cite[p.~$74$]{Moser}. Given $\kappa$, $\lambda\in\mathbb{Z}$ such that $\kappa\leqslant0$ and $\lambda\geqslant1$, let $s=(\infty,s_{\kappa+1},\dots,s_{\lambda-1},\infty)$, $s_i\in\mathbb{N}$, where $\kappa=0$ and $\lambda=1$ represents the symbolic element $(\infty;\infty)$. If we take $\kappa=\infty$ and $\lambda=\infty$, then we identify the above sequences with the elements of $S_\infty$. We also admit half-infinite sequences with $\kappa=-\infty$ and $\lambda<\infty$, or $\kappa>-\infty$ and $\lambda=\infty$. These sequences form the space $\overline{S_\infty}$. For each $k\in\mathbb{N}$, we define the $k$-neighborhood of the element $s^*=(\dots,s^*_{-1},s^*_0;s_1,\dots,s^*_{\lambda-1},\infty)$ as the set of elements $s$ for which 
\[
	s_i=s_i^*, \textnormal{ for } i\in\{-k,\dots,\lambda-1\}, \quad \textnormal{and} \quad s_\lambda\geqslant k.
\]
Defining such neighborhoods for the other types of new elements, we obtain a topology for $\overline{S_\infty}$ extending that of $S_\infty$, and in such a way that $\overline{S_\infty}$ is compact.

The shift map $\sigma$ will not be defined for the entire $\overline{S_\infty}$. Rather, it will be extended to a map $\overline{\sigma}\colon D_\sigma\to R_\sigma$, defined also by~\eqref{02}, i.e. by the shift of the elements of the sequence, but only for the sequences in $D_\sigma$, $R_\sigma\subset\overline{S_\infty}$, given by
\begin{equation}\label{03}
	D_\sigma=\{s\in\overline{S_\infty}\colon s_1\neq\infty\}, \quad R_\sigma=\{s\in\overline{S_\infty}\colon s_0\neq\infty\}.
\end{equation}
It is not hard to see that $\overline{\sigma}$ is a homeomorphism.

One need now to define a map that somewhat plays the role of infinitely many iterates of $F$, i.e. $F^\infty$. To this end, Moser~\cite[p.~$101$]{Moser} defines the \emph{transversal map}, for dimension two, as follows (the case of higher dimensions is postponed to Appendix~\ref{AppA}). Given a homoclinic point $q$, let $R$ be a small enough quadrilateral having $q$ as a corner point, and $W^s(p)$, $W^u(p)$ as two of its sides. Given $r\in R$, let $k=k(r)$ be the smallest positive integer such that $F^k(r)\in R$, if it exists. Let $D_F\subset R$ be the set of points $r$ for which such a $k(r)$ exists. As we shall see, $D_F$ will not be empty. The transversal map $F^\infty\colon D_F\to R$ is defined as $F^\infty(r)=F^{k(r)}(r)$.

\begin{main}[Moser's Theorem]\label{M3}
	Let $F$ satisfy hypothesis~\ref{H}, $p\in\mathcal{U}\setminus\Omega$ be a hyperbolic fixed point of $F$, and $q\in\mathcal{U}\setminus\Omega$ a transverse homoclinic point. Then for every small enough neighborhood $V\subset\mathcal{U}\setminus\Omega$ of $q$ there is $R\subset V$ and a non-empty set $D_F\subset R$ such that $F^\infty\colon D_F\to R$ is well defined. Moreover, there is a homeomorhpism $\tau_\infty\colon S_\infty\to\tau_\infty(S_\infty)\subset D_F$ such that $\tau_\infty(S_\infty)$ is invariant by $F^\infty$, and $F^\infty\circ\tau_\infty=\tau_\infty\circ\sigma$. Furthermore, $\tau_\infty$ can be extended to a homeomorphism $\overline{\tau_\infty}\colon\overline{S_\infty}\to\overline{\tau_\infty(S_\infty)}$ such that $F^\infty\circ\overline{\tau_\infty}=\overline{\tau_\infty}\circ\overline{\sigma}$, if both sides are restricted to $D_\sigma$.
\end{main}

\begin{proof}[Proof of Theorem~\ref{M3}]
Roughly speaking, the proof of Theorem~\ref{M3} for dimension two work as follows, see Figure~\ref{Fig16}. 
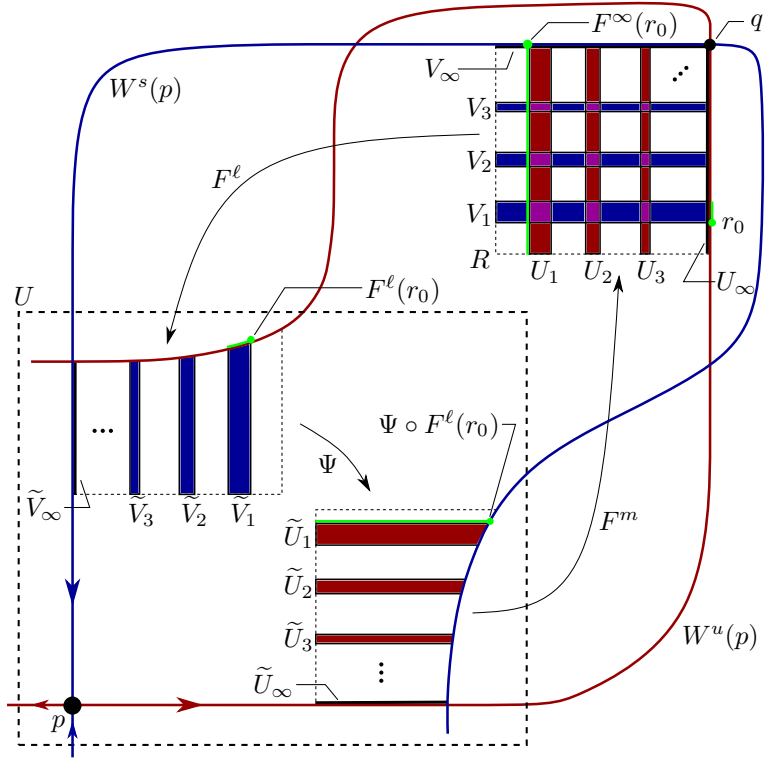
\begin{figure}[ht]
	\begin{center}
		\begin{overpic}[height=10cm]{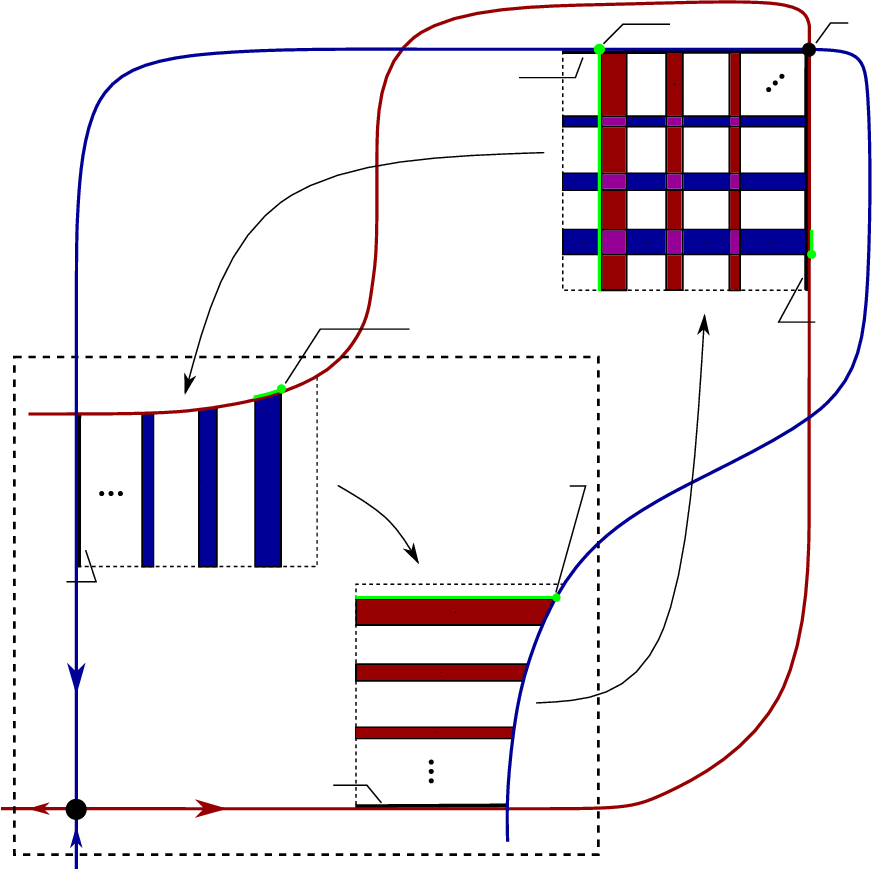} 
			%\begin{overpic}[height=10cm,grid,tics=5]{Fig22.eps} 
			\put(29.5,31){$\widetilde{V}_1$}
			\put(23,31){$\widetilde{V}_2$}
			\put(16,31){$\widetilde{V}_3$}
			\put(2.5,31.75){$\widetilde{V}_\infty$}
			
			\put(36.5,28){$\widetilde{U}_1$}
			\put(36.5,21.5){$\widetilde{U}_2$}
			\put(36.5,14.5){$\widetilde{U}_3$}
			\put(32.5,8.5){$\widetilde{U}_\infty$}
			
			\put(60.5,70.5){$V_1$}
			\put(60.5,77.5){$V_2$}
			\put(60.5,84.5){$V_3$}
			\put(55,89.5){$V_\infty$}
			
			\put(69,63){$U_1$}
			\put(76,63){$U_2$}
			\put(83,63){$U_3$}
			\put(93.5,61.5){$U_\infty$}
			
			\put(6,4){$p$}
			\put(98,96.25){$q$}
			\put(94.25,69.5){$r_0$}
			\put(47.25,60.5){$F^\ell(r_0)$}
			\put(49,42.75){$\Psi\circ F^\ell(r_0)$}
			\put(77,95.5){$F^\infty(r_0)$}
			
			\put(27,75){$F^\ell$}
			\put(41,37.5){$\Psi$}
			\put(78,30){$F^m$}
			
			\put(1,59.5){$U$}
			\put(13.5,87){$W^s(p)$}
			\put(89,15){$W^u(p)$}		
			\put(61,64.5){$R$}	
		\end{overpic}
	\end{center}
	\caption{Illustration of the proof of Theorem~\ref{M3}.}\label{Fig16}
\end{figure}
First we take a small enough neighborhood $U\subset\mathcal{U}\setminus\Omega$ of $p$ in which coordinates $(x_u,x_s)$ of the Stable Manifold Theorem can be assumed. Then we take $\ell$, $m\in\mathbb{N}$ big enough such that $F^\ell(R)$, $F^{-m}(R)\subset U$. Given any set $W\subset U$, it is not hard to see that $F|_U$ contracts $W$ in the direction of $W^s(p)$, while expanding it in the direction of $W^u(p)$. For definiteness, let us suppose $W^s(p)\cap U$ is the vertical axis $x_u=0$, while $W^u(p)\cap U$ is the horizontal one $x_s=0$. Suppose also that $F^\ell(R)$ and $F^{-m}(R)$ are contained in the first quadrant $x_s>0$, $x_u>0$. Let $\{\widetilde{V}_i\}_{i\in\mathbb{N}}\subset F^\ell(R)$ be a sequence of ``vertical'' strips, parallel to $W^s(p)$, and such that $\widetilde{V}_{i+1}$ is on the left-hand side of $\widetilde{V}_i$. For each $i\in\mathbb{N}$, there is $k=k(i)$ big enough, $k(i+1)>k(i)$, such that $F^{k(i)}(\widetilde{V}_i)$ intersects $F^{-m}(R)$ in ``horizontal'' strips $\widetilde{U}_i$, parallel to $W^u(p)$. Let $\Psi\colon\cup_{i\in\mathbb{N}}\widetilde{V}_i\to\cup_{i\in\mathbb{N}}\widetilde{U}_i$ be the map given by $\Psi(\widetilde{V}_i)=F^{k(i)}(\widetilde{V}_i)$. Let $V_i=F^{-\ell}(\widetilde{V}_i)$, $U_i=F^m(\widetilde{U}_i)\subset R$, and notice that $F^\infty:=F^m\circ\Psi\circ F^\ell$ is well defined on the non-empty set $D_F:=F^{-\ell}(\cup_{i\in\mathbb{N}}\widetilde{V}_i)$. The map $\tau_\infty\colon S_\infty\to\tau_\infty(S_\infty)\subset D_F$ takes a sequence $s=(s_m)_{m\in\mathbb{Z}}$ and sends it to a point $r$ such that $(F^\infty)^m(r)\in U_{s_m}$ for every $m\in\mathbb{Z}$.

The extension to $\overline{S_\infty}$ work as follows. First, since $\widetilde{V}_i\to W^s(p)$ and $\widetilde{U}_i\to W^u(p)$ uniformly, we let $\widetilde{V}_\infty\subset W^s(p)$ and $\widetilde{U}_\infty\subset W^u(p)$ be such limits, and consider $V_\infty=F^{-\ell}(\widetilde{V}_\infty)$, $U_\infty=F^m(\widetilde{U}_\infty)$. The sequence $(\infty;s_1,s_2,\dots)$ represents a point $r\in U_\infty$ such that $(F^\infty)^m(r)\in U_{s_m}$. Since $U_\infty\subset W^u(p)$, it follows that we can iterate the point $r$ forward, but not backward, as the latter will converge to $p$ and thus escape. Hence, recall~\eqref{03}, we let $\overline{\sigma}^{-1}\colon R_\sigma\to D_\sigma$ undefined for $s_0=\infty$. Similarly, if we interchange $U_i$ by $V_i$ and $F^\infty$ by its inverse, the sequence $(\dots,s_{-1},s_0;\infty)$ denotes a point lying in $V_{\infty}$ that can be iterated backward, but not forward. In Figure~\ref{Fig16} for instance, the green segment contained in $U_\infty$ represents the set of sequences $(\infty;1,s_2,s_3,\dots)$, while the green point $r_0$ is represented by $(\infty;1,\infty)$. The homoclinic point $q$ is represented by $(\infty;\infty)$. 

Rigorous speaking, since hypothesis~\ref{H} ensures that $F^\ell|_R$, $F^m_R$ and $\Psi$ are $C^1$ diffeomorphisms, the prove of Theorem~\ref{M3} follows by applying the original proof of Moser~\cite[p.~$181$]{Moser}. The proof for higher dimensions follows similarly. For simplicity, we postpone it to Appendix~\ref{AppA}.
\end{proof}

We remark that while $\Lambda_\infty:=\tau_\infty(S_\infty)$ is invariant but not compact, $\overline{\Lambda_\infty}=\overline{\tau_\infty}(\overline{S_\infty})$ is compact but not invariant. Nevertheless, knowing that any sequence of the form $(\infty,\dots,s_0;s_1,\dots\infty)$ is a homoclinic point, it follows that the set of homoclinic points of $p$ is dense in $\overline{\Lambda_\infty}$, see~\cite[Corollary~$3.8$]{Moser}.

\subsection{Chaos}

We now use Devaney's notion of chaos~\cite[Chapter $8$]{Devaney} to formulate the chaotic dynamics provided by the existence of transversal homoclinic points.

\begin{definition}
	Let $(\Lambda,d)$ be a metric space and $f\colon\Lambda\to\Lambda$ a continuous map. We say that $f$ is \emph{chaotic in the sense of Devaney} if the following statements holds.
	\begin{enumerate}[label=(\alph*)]
		\item The set of periodic points of $f$ is dense in $\Lambda$;
		\item $f$ is transitive in $\Lambda$, i.e. given any two non-empty open sets $U_1$, $U_2\subset \Lambda$, there is $n>0$ such that $f^n(U_1)\cap U_2\neq\emptyset$.
	\end{enumerate}
\end{definition}

In his original definition of chaos, Devaney also included the following condition in addition with $(a)$ and $(b)$:
\begin{enumerate}[label=(\alph*)]
	\addtocounter{enumi}{2}
	\item $f$ has sensitive dependence in $\Lambda$, i.e. there is $\beta>0$ such that for any $x\in\Lambda$ and any neighborhood $U\subset\Lambda$ of $x$, there are $r\in U$ and $n>0$ such that $d\big(f^n(x),f^n(r)\big)>\beta$.
\end{enumerate} 
However, Banks et al~\cite{Banks} proved that this third condition is redundant, i.e. $(a)$ and $(b)$ hold, then $(c)$ holds. Moreover, consider the following condition.
\begin{enumerate}
	\item[$(b')$] $f$ has a dense orbit in $\Lambda$.
\end{enumerate} 
We notice that if $(b')$ holds, then $(b)$ holds. In fact, they are equivalent~\cite[Lemma $4.3.4$]{ViaOli2016}. Under the hypotheses of Theorems~\ref{M2} and~\ref{M3}, it is clear from the conjugation of $F^k$ and $F^\infty$ with the Bernoulli Shift that statements $(a)$ and $(b')$ holds, and thus both maps are chaotic in the sense of Devaney. Hence, we have the following result.

\begin{main}\label{M4}
	Let $F$ satisfy hypothesis~\ref{H}, $p\in\mathcal{U}\setminus\Omega$ be a hyperbolic fixed point of $F$, $q\in\mathcal{U}\setminus\Omega$ a transverse homoclinic point, and $k\in\mathbb{N}$ as in Theorem~\ref{M2}. Then $F^k$ and $F^\infty$, restricted to their invariant sets, are chaotic in the sense of Devaney.
\end{main}

\section{Applications of the results}\label{Sec4}

In this section we apply the results obtained in Section~\ref{Sec3} in piecewise smooth vector fields. In particular, we show that they can be lifted to the vector field itself.

Let $Z=(X_1,\dots,X_M,\Sigma)$ be a piecewise smooth vector field defined by~\eqref{01}, with each $C^1$-component $X_i=X_i(x)$ autonomous. For each $i\in\{1,\dots,M\}$ consider a non-autonomous perturbation $\mathcal{X}_j(t,x)=X_j(x)+\varepsilon Y_j(t,x)$, $\varepsilon>0$ small enough, $Y_j$ of class $C^1$ in $(t,x)$, and also $T$-periodic in $t$, $T>0$. Knowing that $\Sigma=\cup_{i=1}^N\Sigma_i$ with $\Sigma_i=h^{-1}(0)$, we let $\mathcal{H}_i(t,x):=h_i(x)$, $\Omega_i:=\mathcal{H}_i^{-1}(0)$, and $\Omega=\cup_{i=1}^N\Omega_i$. Finally, consider the piecewise smooth vector field $\mathcal{Z}=(\mathcal{X}_1,\dots,\mathcal{X}_n,\Omega)$. Since $\mathcal{Z}$ is non-autonomous, it is preferable to work in the extended phase space $\mathbb{S}^1\times\mathbb{R}^n$, with $\mathbb{S}^1:=\mathbb{R}/T\mathbb{Z}$, where $\mathcal{Z}$ is autonomous.

Suppose $\mathcal{Z}$ is free of sliding motion in an open set $\mathcal{U}\subset\mathbb{S}^1\times\mathbb{R}^n$, and let $\varphi(\tau,t,x)$ be its solution satisfying $\varphi(0,t,x)=(t,x)$. Since $\mathcal{Z}$ is $T$-periodic in $t$, for each $t\in\mathbb{S}^1$ we can consider its respective time-$T$-map $\Phi^T_t(x):=\varphi(T,t,x)$. Given $t_1$, $t_2\in\mathbb{S}^1$, it is not hard to see that $\Phi^T_{t_1}$ and $\Phi^T_{t_2}$ are topologically conjugated, with the conjugation given by $\Phi^{t_2-t_1}_{t_1}(x):=\varphi(t_2-t_1,t_1,x)$, with inverse $\Phi^{t_1-t_2}_{t_2}$. Moreover, it follows from Lemma~\ref{P2} that for every $(t_1,x)\not\in\Omega$ such that $\Phi^{t_2-t_1}_{t_1}(x)\not\in\Omega$, we have that $\Phi^{t_2-t_1}_{t_1}$ is a $C^1$-diffeomorphism in a neighborhood of such point. Hence, $\Phi^{t_2-t_1}_{t_1}$ is a $C^1$-differomophism except in finitely many sub-manifolds of codimension one. In particular, $\Phi^T_{t_1}$ and $\Phi^T_{t_2}$ are also $C^1$-conjugated excepted in finitely many sub-manifolds of codimension one.

\begin{main}\label{M5}
	Let $\mathcal{Z}$ be as above. If $\Phi^T_{t_0}$ has a transversal homoclinic point in $\mathcal{U}\setminus\Omega$, then for all $t\in\mathbb{S}^1$ except finitely many, $\Phi^T_t$ also has a transverse homoclinic point.
\end{main}

\begin{proof}
	Let $q_0=(t_0,x_0)$ be the homoclinic point of $\Phi^T_{t_0}$, and for each $t\in\mathbb{S}^1$ let $q_t=\Phi^{t-t_0}_{t_0}(x_0)$. Since $q_0\not\in\Omega$, it follows from Lemma~\ref{P2} that if $q_t\not\in\Omega$, then $\Phi^{t-t_0}_{t_0}$ is a $C^1$-diffeomorphism in a neighborhood of $q_0$. From Lemma~\ref{P1} we have that there are at most finitely many $t_1,\dots,t_k$ such that $q_{t_i}\in\Omega$, proving the theorem.
\end{proof}

In particular, it follows from Theorem~\ref{M5} that if $\mathcal{Z}$ has a hyperbolic $T$-periodic orbit $\gamma\subset\mathcal{U}\setminus\Omega$, and for some level $t=t_0$ the hyperbolic point $p_0:=\gamma\cap\{t=t_0\}$ has a transversal homoclinic point in relation to the map $\Phi^T_{t_0}$, then the stable and unstable manifolds of $\gamma$ also get entangled, similarly to a lift of Figure~\ref{Fig6}$(b)$. See Figure~\ref{Fig13} and also~\cite[Figure~$3.40$]{Arrow}.
\begin{figure}[ht]
	\begin{center}
		\begin{overpic}[height=7.5cm]{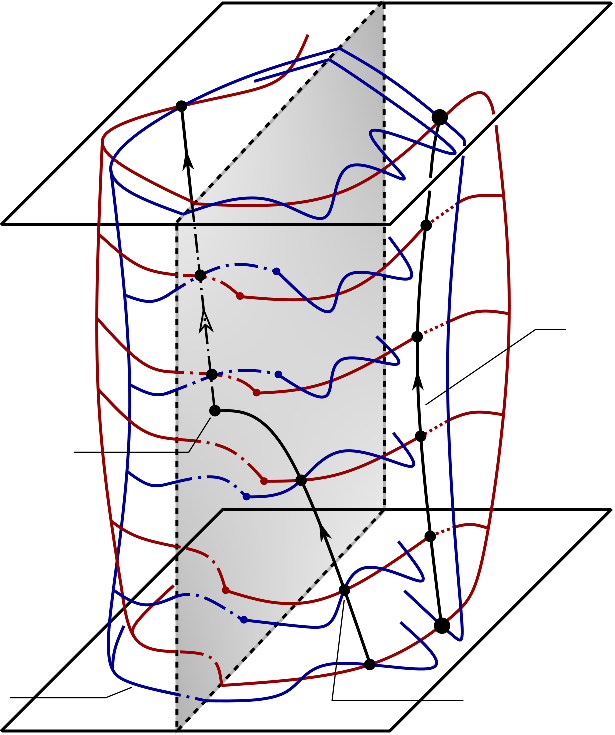} 
		%\begin{overpic}[height=7.5cm,grid,tics=5]{Fig16.eps}
			\put(78,54.25){$\gamma$}
			\put(19.5,1){$\Omega$}
			\put(81,24){$t=0$}
			\put(81,93){$t=T$}
			\put(69,38){$W^u(\gamma)$}
			\put(-13.5,3.75){$W^s(\gamma)$}
			\put(64,4){$q_0$}
			\put(5.5,37.5){$q_t$}
		\end{overpic}
	\end{center}
	\caption{Illustration of the entanglement of $\gamma$.}\label{Fig13}
\end{figure}
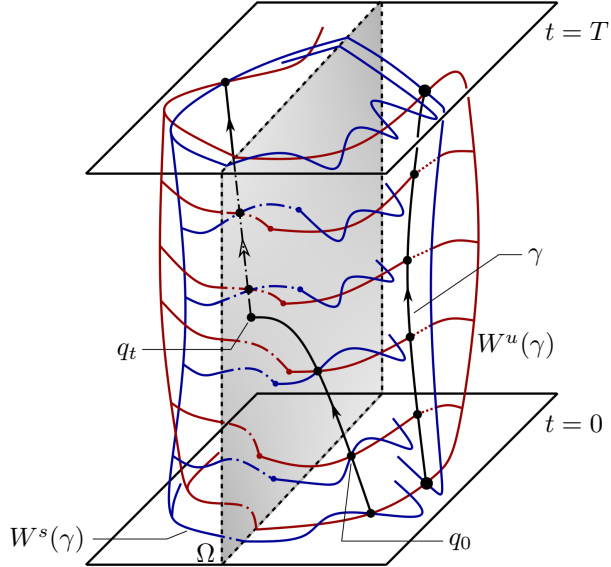
We finish this section with an example.

\begin{example}
	Consider the planar piecewise smooth system $\mathcal{Z}=(\mathcal{X}_1,\mathcal{X}_2,\Omega)$ given by $\Omega=\{(x,y)\in\mathbb{R}^2\colon y=0\}$ and
	\[
		\mathcal{Z}(t,x,y;\varepsilon)=\left\{\begin{array}{ll}
			\displaystyle \left(-\frac{4}{5}y,\; x+\varepsilon\sin t\right), & \text{if } y\geqslant0, 
			\vspace{0.1cm} \\
			\displaystyle \left(\frac{1}{3}-\frac{1}{3}y-\frac{1}{3}y^2,\; x+\varepsilon\sin t\right), & \text{if } y\leqslant0.
		\end{array}\right.
	\]
	It was proved in~\cite[Section~$4$]{Kupper} that $\mathcal{Z}(t,x,y;0)$ has a hyperbolic saddle point $p=(0,-(1+\sqrt{5})/2)$	with a piecewise smooth homoclinic point. Moreover, it is proved by Melnikov method that $\mathcal{Z}(t,x,y;\varepsilon)$ has a transverse homoclinic point for $\varepsilon>0$ small enough. Therefore, it follows from Theorem~\ref{M4} that $\mathcal{Z}_\varepsilon$ is chaotic. Furthermore, we have from Theorem~\ref{M5} that the hyperbolic orbit $\gamma_\varepsilon$ obtained from $p$ has an entanglement between its stable and unstable manifolds.
\end{example}

\section{Proof of the technical lemmas}\label{Sec5}

\begin{proof}[Proof of Lemma~\ref{P1}]
	Suppose that $\{\varphi(t,x):t\in[0,T]\}\cap\Sigma$ is not finite. Then there is an increasing sequence $(t_k)\subset[0,T]$ such that $\varphi(t_k,x)\in\Sigma$ for every $k\in\mathbb{N}$. Let 
	\[
		t^*:=\lim\limits_{k\to\infty}t_k=\sup\{t_k\colon k\in\mathbb{N}\}<T,
	\]
	and observe that $t_k-t_{k-1}\to0$ as $k\to\infty$. Let $x_k=\varphi(t_k,x)$, $k\in\mathbb{N}$, and $x^*=\varphi(t^*,x)$. Since $\Sigma$ is closed, it follows that $x^*\in\Sigma$ and thus by hypothesis it is also a crossing point.	Observe that we can take a local system of coordinates in a neighborhood $W\subset\mathbb{R}^n$ of $x^*$ such that in $W$ we have $\Sigma=h^{-1}(\{0\})$ and $Z=(X^+,X^-)$, with $X^\pm$ defined in a neighborhood of the region $\Sigma^\pm=\{x\in\mathbb{R}^n\colon \pm h(x)\geqslant0\}$,	for some $C^1$-function $h\colon W\to\mathbb{R}$. Suppose for definiteness that the trajectory of $Z$ through $x^*$ enters $\Sigma^-$. Restricting $W$ if necessary, we can assume that the trajectory of $Z$ through every $r\in W\cap\Sigma$ also enters $\Sigma^-$, see Figure~\ref{Fig5}.
	\begin{figure}[ht]
		\begin{center}
			\begin{overpic}[height=3cm]{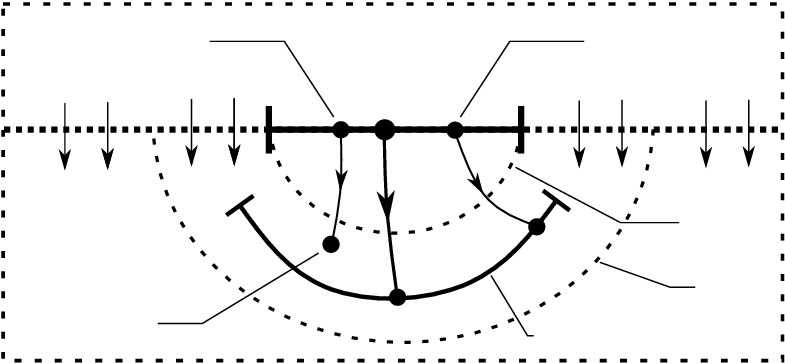} 
			%\begin{overpic}[height=3cm,grid,tics=5]{Fig6x.eps} 
				\put(100.5,44){$W$}
				\put(1,31){$\Sigma$}
				\put(5,13){$\Sigma^-$}
				\put(48,40){$\Sigma^+$}
				\put(48,32){$x^*$}
				\put(20.5,40){$x_k$}
				\put(8.5,4){$x_{k+1}$}
				\put(75.5,39.75){$r$}
				\put(86,16){$A$}
				\put(89,8){$U$}
				\put(68.5,2.5){$\varphi^-(\eta,A\cap\Sigma)$}
			\end{overpic}
		\end{center}
		\caption{Illustration of the proof of Lemma~\ref{P1}.}\label{Fig5}
	\end{figure}
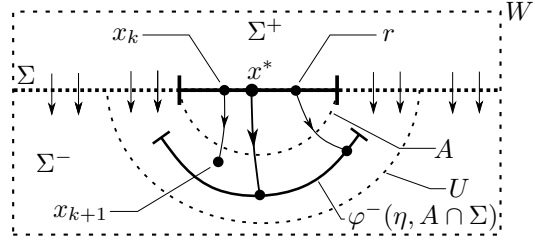
	Let $\varphi^-(\tau,r)$ be the solution of $X^-$ with initial condition $\varphi^-(0,r)=r$. Given a neighborhood $U\subset W$ of $x^*$, it follows from the continuous dependence on the initial conditions~\cite[Theorems $8$ and $9$]{Andronov} that there is a neighborhood $A\subset U$ of $x^*$ and $\eta>0$ such that if $(\tau,r)\in[0,\eta]\times( A\cap\Sigma^-)$, then $\varphi^-(\tau,r)\in U\cap\Sigma^-$.	Restricting $A$ if necessary, we have from the Flow Box Theorem~\cite[Theorem~$1.12$]{DumLliArt2006} that if $(\tau,r)\in(0,\eta]\times (A\cap\Sigma)$, then $\varphi^-(\tau,r)\in (U\cap\Sigma^-)\setminus\Sigma$. Since $x_k\to x^*$ and $t_k-t_{k-1}\to0^+$, it follows that for $k\in\mathbb{N}$ big enough we have $x_k\in A\cap\Sigma$ and $t_{k+1}-t_k<\eta$, which implies that $x_{k+1}\in(U\cap\Sigma^-)\setminus\Sigma$, contradicting the fact that $x_{k+1}\in\Sigma$ is a crossing point.
\end{proof}

\begin{proof}[Proof of Lemma~\ref{P2}]
	It follows from Lemma~\ref{P1} that there are at most finitely many $t_1,\dots,t_k\in(0,T)$ such that $\varphi(t_i,x)\in\Sigma$. We assume for the moment $k=1$. In this case, we can also assume that there exists a $C^1$-function $h\colon\mathbb{R}^n\to\mathbb{R}$ such that $\Sigma=h^{-1}(\{0\})$ and $Z=(X^+,X^-)$, with $X^\pm$ defined in a neighborhood of the region $\Sigma^\pm=\{x\in\mathbb{R}^n\colon \pm h(x)\geqslant0\}$, with the corresponding solution denoted by $\varphi^\pm(\tau,r)$. Suppose for definiteness that $x\in\Sigma^+$ and $\varphi(T,x)\in\Sigma^-$, see Figure~\ref{Fig1}. 
	\begin{figure}[ht]
		\begin{center}
			\begin{overpic}[width=4cm]{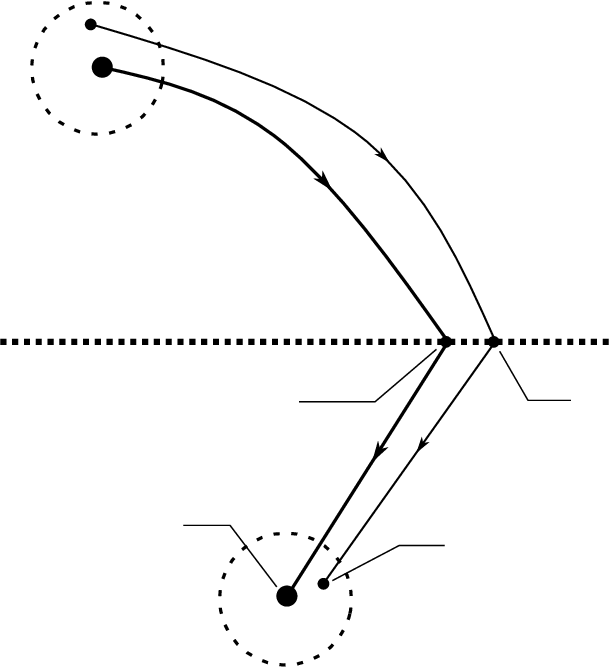} 
			%\begin{overpic}[width=4cm,grid,tics=5]{Fig1.eps} 
				\put(-2,50){$\Sigma$}
				\put(20,77){$A$}
				\put(11,84){$x$}
				\put(9,92){$r$}
				\put(68,15.5){$\Phi^T(r)$}
				\put(5,19){$\Phi^{T}(x)$}
				\put(87,38.5){$\varphi(\tau(r),r)$}
				\put(18,37.75){$\varphi(t_1,x)$}
				\put(52,-1){$\Phi^T(A)$}
			\end{overpic}
		\end{center}
		\caption{Illustration of the proof of Lemma~\ref{P2}.}\label{Fig1}
	\end{figure}
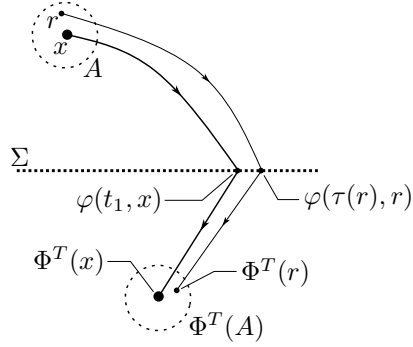	
	Since $X^+$ is defined in a neighborhood of $\Sigma^+$, there exists $\eta>0$ sufficiently small such that $\varphi^+(\tau,r)$ is well defined for $\tau\in I:=(-\eta,t_1+\eta)$. Hence, it follows from the continuous dependence on the initial conditions that there is a neighborhood $A\subset\mathbb{R}^n\setminus\Sigma$ of $x$ such that $\varphi^+(\tau,r)$ is well defined for every $(\tau,r)\in I\times A$. Let $\Psi\colon I\times A\to\mathbb{R}$ be given by $\Psi(t,r)=h(\varphi^+(t,r))$. Notice that $\Psi$ is well defined and of class $C^1$ in $I\times A$. Moreover, 
	\begin{equation}\label{1}
		\partial_t\Psi(t,r)=\left<\nabla h(\varphi^+(t,r)),X^+(\varphi^+(t,r))\right>,
	\end{equation}
	where we have used that $\partial\varphi^+/\partial t=X^+\circ\varphi^+$. Thus, equation~\eqref{1} in addition with conditions~\ref{ii} and~\ref{iii} from definition of crossing points lead us to $\Psi(t_1,x)=0$ and $\partial_t\Psi(t_1,x)\neq0$. Therefore, it follows from the Implicit Function Theorem that there is a $C^1$-function $\tau\colon A\to\mathbb{R}$, with $\tau(x)=t_1$, such that $\Psi(\tau(r),r)\equiv0$. Observe that $\tau(r)$ is exactly the time necessary for a point $r\in A$ to reach $\Sigma$. Restricting $A$ if necessary, we can assume that there is $\tau_0>0$ such that $\tau(r)\leqslant\tau_0<T$ for every $r\in A$. This implies that, for $r\in A$, the map $\Phi^T$ can be written as
	\[
		\Phi^T(r)=\varphi^-\big(T-\tau(r),\varphi^+(\tau(r),r)\big).
	\]
	In particular, $\Phi^T$ is $C^1$ when restricted to $A$. By reversing the time, we observe that the endpoint $\Phi^{T}(x)$ satisfy the same hypothesis as $x$ and thus we can also prove that $\Phi^{-T}$ is $C^1$ in a neighborhood of $\Phi^{T}(x)$. This concludes the proof, since $\Phi^{T}$ and $\Phi^{-T}$ are inverses one of each other.
	
	In order to prove the general case (i.e., $k>1$), let us set $T_0=0$ and $T_k=T$, and for each $i\in\{1,\dots,k-1\}$, choose $T_i\in(t_i,t_{i+1})$. Also, let $x_i=\varphi(T_i,x)$ for $i\in\{0,\dots,k-1\}$. Analogously to the previous reasoning, the map $\Phi^{T_{i+1}-T_i}$ is a $C^1$-difeomorphism in a neighborhood of $x_i$, ${i\in\{0,\dots,k-1\}}$. The proof now follows from the fact that $\Phi^T=\Phi^{T_k-T_{k-1}}\circ\dots\circ\Phi^{T_1-T_0}$.
\end{proof}

\begin{proof}[Proof of Lemma~\ref{P3}]
	If $\{\varphi(t,x):t\in[0,T]\}$ never intersects $\Sigma$ then there is nothing to prove. On the other hand, suppose $\{\varphi(t,x):t\in[0,T]\}\cap\Sigma\neq\emptyset$. If $x\not\in\Sigma$ and $\varphi(T,x)\not\in\Sigma$, then the result follows from Lemmas~\ref{P1} and~\ref{P2}. Hence we need to work the cases in which $x\in\Sigma$ or $\varphi(T,x)\in\Sigma$.
	
	First we suppose $x\not\in\Sigma$ and $\varphi(T,x)\in\Sigma$. For simplicity, we can also assume $\varphi(t,x)\not\in\Sigma$ for every $t\in(0,T)$. Similarly to the proof of Lemma~\ref{P1}, there is a neighborhood $A\subset\mathbb{R}^n\setminus\Sigma$ of $x$ and a $C^1$-function $\tau\colon A\to\mathbb{R}$, with $\tau(x)=T$, such that $\varphi(\tau(r),r)\in\Sigma_c$ for every $r\in A$. Hence, $\Phi^T|_A$ can be written as
	\[\Phi^T|_A(r)=\left\{\begin{array}{ll}
		\varphi^-\big(T-\tau(r),\varphi^+(\tau(r),r)\big) & \text{if } \tau(r)\leqslant T, \\
		\varphi^+(T,r) & \text{if } \tau(r)\geqslant T,
	\end{array}\right.\]
	where $\varphi^\pm(\tau,r)$ are the solution of $X^\pm$, see Figure~\ref{Fig2}$(a)$.
	\begin{figure}[ht]
		\begin{center}
			\begin{minipage}{6.25cm}
				\begin{center} 
					\begin{overpic}[width=5cm]{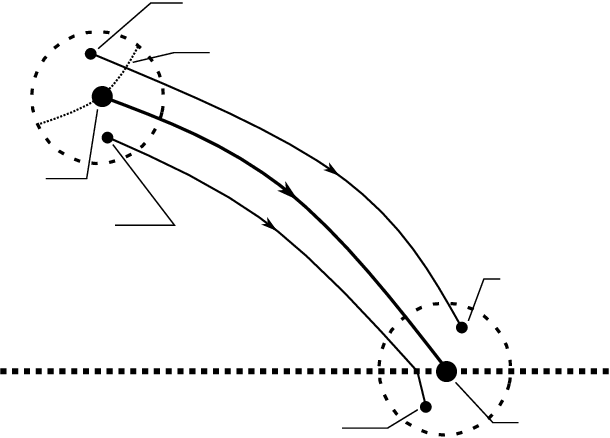} 
					%\begin{overpic}[width=5cm,grid,tics=5]{Fig2.eps} 
						\put(-2,13){$\Sigma$}
						\put(31,69.5){$r_1$}
						\put(36,61.5){$\tau^{-1}(T)$}
						\put(2.5,40.5){$x$}
						\put(11,33.5){$r_2$}
						\put(83,23.5){$\Phi^T(r_1)$}
						\put(86,0){$\Phi^T(x)$}
						\put(32,0){$\Phi^T(r_2)$}
					\end{overpic}
					
					$\;$
					
					$(a)$
				\end{center}
			\end{minipage}
			\begin{minipage}{6.25cm}
				\begin{center} 
					\begin{overpic}[width=5cm]{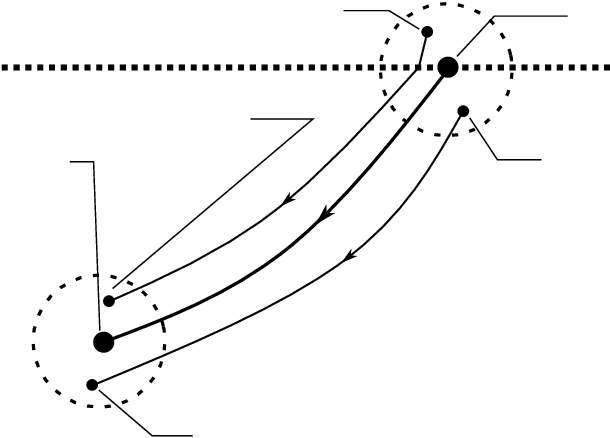} 
					%\begin{overpic}[width=5cm,grid,tics=5]{Fig3.eps} 
						\put(-2,63){$\Sigma$}
						\put(49,69){$r_1$}
						\put(94,67.5){$x$}
						\put(90,44){$r_2$}
						\put(17,50){$\Phi^T(r_1)$}
						\put(33,-2){$\Phi^T(r_2)$}
						\put(-9.5,43){$\Phi^T(x)$}
					\end{overpic}
					
					$\;$
					
					$(b)$
				\end{center}
			\end{minipage}
		\end{center}
		\caption{Illustration of the proof of Lemma~\ref{P3}.}\label{Fig2}
	\end{figure}
	Thus, we have from the Pasting Lemma~\cite[Theorem $18.3$]{Munkres} that $\Phi^T$ is continuous in $A$. The case $x\in\Sigma$ and $\varphi(T,x)\not\in\Sigma$ follows similarly, with the difference that $\tau(x)=0$, leading to
	\[\Phi^T|_A(r)=\left\{\begin{array}{ll}
		\varphi^-\big(T-\tau(r),\varphi^+(\tau(r),r)\big) & \text{if } \tau(r)\geqslant 0, \\
		\varphi^-(T,r) & \text{if } \tau(r)\leqslant 0,
	\end{array}\right.\]
	where $A$ is restricted so that $|\tau(r)|\leqslant\tau_0<T$ for every $r\in A$, see Figure~\ref{Fig2}$(b)$. The case where $x\in\Sigma$ and $\varphi(T,x)\in\Sigma$ follows similarly. Observe that $\Phi^{-T}$ satisfies similar hypothesis and thus can also be proved continuous in a neighborhood of $\varphi(T,x)$. The result now follows from the fact that $\Phi^{-T}$ is the inverse of $\Phi^T$. The general case in which there are finitely many $t_1,\dots,t_k\in(0,T)$ such that $\varphi(t_k,x)\in\Sigma$ is analogous to end of the proof of Lemma~\ref{P2}.
\end{proof}

\appendix

\section{Moser's Theorem for higher dimensions}\label{AppA}

We now comment on the extension of Moser's Theorem for higher dimensions. The main difference is that we use the $\lambda$-Lemma instead of the technical Lemma~\cite[p.~$182$]{Moser}. More precisely, let $U\subset\mathcal{U}\setminus\Omega$ be a small enough neighborhood of the hyperbolic fixed point $p\in\mathcal{U}\setminus\Omega$ for which we can assume coordinates $(x_u,x_s)$ from the Stable Manifold Theorem. For definiteness, we suppose $W^s(p)\cap U$ is given by $x_u=0$, while $W^u(p)\cap U$ is given by $x_s=0$. See Figure~\ref{Fig4}.
\begin{figure}[ht]
	\begin{center}
		\begin{overpic}[width=6cm]{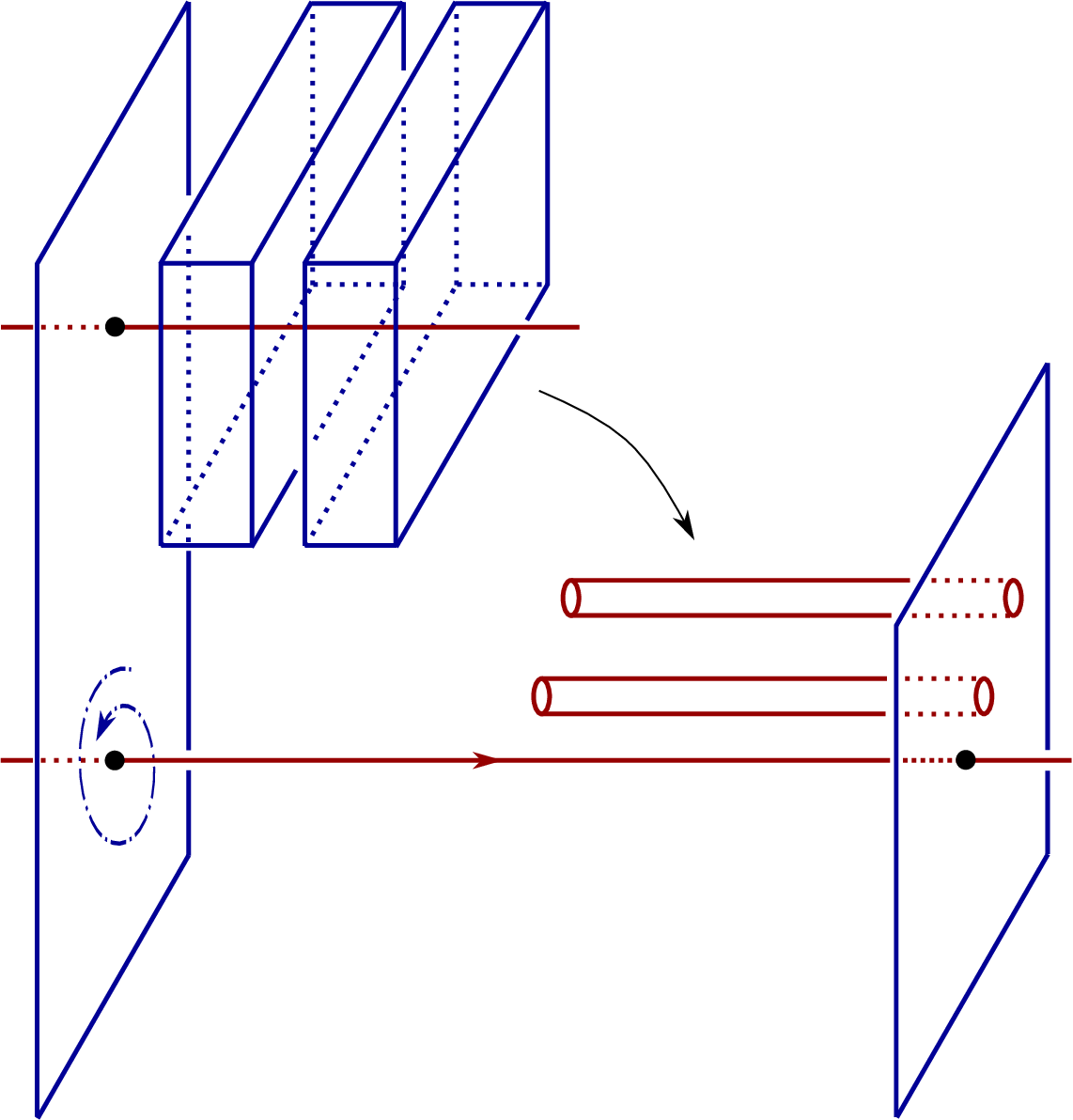} 
		%\begin{overpic}[width=6cm,grid,tics=5]{Fig23.eps} 
			\put(10,28){$p$}
			\put(50,27){$W^u(p)$}
			\put(-5,90){$W^s(p)$}
			\put(9,66.5){$q_\ell$}
			\put(86,28){$q_m$}
			\put(56.5,60.25){$\Psi$}
			
			\put(28,44.5){$\widetilde{V}_1$}
			\put(18,44.5){$\widetilde{V}_2$}
			
			\put(44,44.5){$\widetilde{U}_1$}
			\put(41,35.5){$\widetilde{U}_2$}
		\end{overpic}
	\end{center}
\caption{Illustration of the proof of Theorem~\ref{M3} for higher dimensions.}\label{Fig4}
\end{figure}
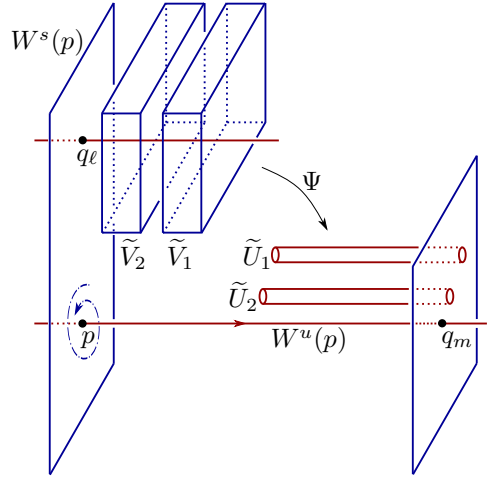
Let $R\subset\mathcal{U}\setminus\Omega$ be a small enough neighborhood of the transverse hyperbolic point $q\in\mathcal{U}\setminus\Omega$. Let $\ell$, $m\in\mathbb{N}$ be big enough such that $F^\ell(R)$, $F^{-m}(R)\subset U$. Let $q_\ell=F^\ell(q)$ and $q_m=F^{-m}(q)$.

Let $B_u\subset\mathbb{R}^u$ be a small enough neighborhood of the origin, and $\Gamma\colon B_u\to\mathbb{R}^u\times\mathbb{R}^s$ a local parametrization of $W^u(p)$ around $q_\ell$, with $\Gamma=(\gamma_u,\gamma_s)$ and $\Gamma(0)=q_\ell$. It follows from the $\lambda$-Lemma that if we take $\ell$ big enough, then $D\gamma_u(0)$ is non-singular. Hence, it follows from the Local Immersion Theorem~\cite[p.~$15$]{GuiPol1974} that there is a local $C^1$-change of coordinates around $q_\ell$ such that in these new coordinates $(w_u,w_s)$, $W^u(p)$ is given by $w_s=0$. We then use this new coordinate system to create ``vertical strips'' $\widetilde{V}_i$ that converges to $W^s(p)$. This can be done by taking a sequence $(r_i)\subset W^u(p)$, $r_i\to q_\ell$, $r_i=(s_i,u_i)$, and defining $\widetilde{V}_i=\{(w_s,w_u)\colon w_u\in B(u_i,\varepsilon_i)\}$, with $\varepsilon_i\to0^+$, and $B(r_i,\varepsilon_i)$ denoting the $u$-dimensional open ball centered at $u_i$ and of radius $\varepsilon_i>0$. We now notice that the iterations of $\widetilde{V}_i$ are contracted in the stable direction and expanded in the unstable one. This can be made precise with the definition of the stable (resp. unstable) distance $d_s$ (resp. $d_u$), see~\cite[p.~$230$]{Newhouse}. The proof now follows similarly to the two-dimensional case, with the definition of the ``horizontal strips'' $\widetilde{U}_i$, and then the map $\Psi$. The last observation to be made is that if $W^s(p)$ is of codimension one, then it separates $U$ in the regions $x_u>0$ and $x_u<0$, and thus a suitable side must be chosen for the sequence $(r_i)\subset W^u(p)$ (the one containing $q_m$). In the particular case dimension two, both $W^{s,u}(p)$ are of codimension one, and thus quadrant must be chosen.

\section*{Acknowledgments}

This work is supported by  S\~ao Paulo Research Foundation (FAPESP), grants 2021/01799-9, 2023/02959-5 and 20\-24/15612-6; CNPq, grants 401974/2025-1 and 307706/2023-0;  Capes, grant 88881.179491/2025-01, and Agence Nationale de la Recherche (ANR), project ANR-23-CE40-0028.

\end{document}